\documentclass[11 pt]{amsart}
\usepackage{amsmath}
\usepackage{amsthm}
\usepackage{amsfonts}
\usepackage{amsxtra}
\usepackage[all]{xy}
\usepackage{verbatim}
\usepackage{color}
\usepackage{amssymb}

\theoremstyle{plain}   
\newtheorem{thm}{Theorem}[section] 
\newtheorem{cor}[thm]{Corollary}
\newtheorem{lemma}[thm]{Lemma}
\newtheorem{prop}[thm]{Proposition}

\newtheorem{thrm}[thm]{Theorem}

\theoremstyle{definition}

\newtheorem{defn}[thm]{Definition}

\theoremstyle{remark}

\newtheorem{rem}[thm]{Remark}
\newtheorem{ex}[thm]{Example}

\makeatletter
\let\c@equation\c@thm
\numberwithin{equation}{section}
\makeatother

\newcommand{\N}{\mathcal{N}}

\newcommand{\M}{\mathcal{M}}

\newcommand{\C}{\mathcal{C}}
\newcommand{\V}{\mathcal{V}}

\newcommand{\A}{\mathcal{A}}
\newcommand{\B}{\mathcal{B}}

\newcommand{\pcC}{\mathit{C}}
\newcommand{\pcD}{\mathit{D}}

\newcommand{\ab}{\mathcal{A}b}
\newcommand{\Spec}{\mathcal{S}pec}
\newcommand{\K}{\mathbb{K}}

\newcommand{\D}{\mathcal{D}}

\newcommand{\sma}{\wedge}
\newcommand{\fto}{\xrightarrow}
\newcommand{\xto}{\xrightarrow}
\newcommand{\perm}{Perm}
\newcommand{\permmc}{\underline{Perm}}
\newcommand{\kobj}{\Phi}
\newcommand{\PC}{\mathbf{PC}}
\newcommand{\col}{\colon}
\newcommand{\BGcat}{G\mathcal{E}}
\newcommand{\BGspec}{G\mathcal{B}}

\newcommand{\op}{\text{op}}
\newcommand{\cat}{Cat}
\newcommand{\HoSpec}{\mathrm{Ho}\Spec}
\newcommand{\comp}{\textrm{comp}}
\newcommand{\id}{\textrm{id}}
\newcommand{\finset}{\mathcal{F}}

\DeclareMathOperator{\Fun}{Fun}
\newcommand{\eqmach}{K_G}

\title{Constructing equivariant spectra via categorical Mackey functors}
\author{Anna Marie Bohmann}
\address{
Department of Mathematics,
Northwestern University,
Evanston, IL 60208 USA
}
\email{bohmann@math.northwestern.edu}
\author{Ang\'elica M. Osorno}
\address{
Department of Mathematics,
Reed College,
Portland, OR 97202 USA
}
\email{aosorno@reed.edu}

\let\oldmarginpar\marginpar
\renewcommand\marginpar[1]{\-\oldmarginpar[\raggedleft\footnotesize #1]%
{\raggedright\footnotesize #1}}

\begin{document}

\begin{abstract}
We give a functorial construction of equivariant spectra from a  generalized version of Mackey functors in categories.  This construction relies on the recent description of the category of equivariant spectra due to Guillou and May.  The key element of our construction is a spectrally-enriched functor from a spectrally-enriched version of permutative categories to the category of spectra that is built using an appropriate version of $K$-theory.  As applications of our general construction, we produce a new functorial construction of equivariant Eilenberg--MacLane spectra for Mackey functors and for suspension spectra for finite $G$-sets.
\end{abstract}

\thanks{ This material is partially based on work supported by the National Science Foundation under Grant No. 0932078 000 while the authors were in residence at the Mathematical Sciences Research Institute in Berkeley, California, during the Spring 2014 semester. }

\maketitle
\section{Introduction}

Spectra are the main objects of study in stable homotopy theory and Eilenberg--MacLane spectra are arguably the simplest kind of spectra. These are the spectra that represent ordinary cohomology with coefficients in abelian groups; they have precisely one nonzero homotopy group.   Their basic role in the field has led topologists to a variety of constructions of Eilenberg--MacLane spectra for  abelian groups.  In particular, abelian groups are one  of the simplest forms of input for all known ways of building spectra from algebraic or categorical data.
 For example, classical infinite loop space theory builds infinite loop spaces, and hence spectra, from symmetric monoidal categories.  An abelian group can be regarded as a ``discrete'' symmetric monoidal category, with the elements of the group as objects and only identity morphisms. Applying an infinite loop space machine to this simple input produces the Eilenberg--MacLane spectrum for the abelian group in question.

The natural equivariant version of ordinary cohomology is \emph{Bredon cohomology,} whose coefficients are given by \emph{Mackey functors}.  These cohomology theories also have representing Eilenberg--MacLane spectra which are again particularly simple examples of spectra.  These spectra can be constructed via obstruction theory techniques or via the abelianization method of dos Santos and Nie \cite{dosSantosNie}.  However, these constructions are all particular to Eilenberg--MacLane spectra.  Furthermore, while machines for constructing general equivariant infinite loop spaces exist \cite{CostWan,Shimakawa,LMS}, Mackey functors cannot transparently be regarded as input for these machines as they can in the nonequivariant case. In this paper, we provide a construction of equivariant spectra that bridges this gap: we produce a genuine equivariant spectrum from a natural categorical generalization of Mackey functors, of which the simplest examples are Mackey functors themselves.

Our construction emphasizes the view that equivariant objects should be understood via their ``fixed points.'' This idea has its roots in the study of group actions on spaces, where one of the most frequently used constructions is that of the \emph{fixed point space}.  If $X$ is a space with a $G$-action, the $G$-fixed points of $X$ is the space $X^G$ consisting of points of $X$ fixed by all elements of $G$. In general, the space $X^G$ forgets a significant amount of information about $X$. 
 However, by considering the collection of fixed point spaces $X^H$  for all subgroups $H$ of $G$,  we can recover $X$ as the fixed points under the identity subgroup.  Moreover, there are natural inclusion and conjugation maps relating the fixed point spaces $X^H$ as $H$ varies, which mean that the collection of fixed point spaces has the structure of a presheaf on the \emph{orbit category} of $G$.  Up to homotopy, any such presheaf can be realized as the fixed points of a space $X$.   This is the content of Elmendorf's theorem \cite{Elmendorf83} and its extension to model categories by Piacenza \cite{Piacenza91}.  Thus, up to homotopy, it suffices to work with the category of these presheaves, and we may thus consider a $G$-space solely in terms of its presheaf of fixed point spaces.

In the stable world, the idea of viewing equivariant objects in terms of their fixed points is at the heart of one of the two main approaches to genuine $G$-spectra. 
Genuine $G$-spectra are more than just spectra with an action of a group $G$: they are required to have extra structure which allows for characteristic classes for equivariant vector bundles.  The most common way of encoding this structure is by requiring a genuine $G$-spectrum to have deloopings by $G$-representation spheres in addition to spheres with a trivial $G$-action.  However, one can instead  require ``transfers'' or ``wrong-way'' maps between fixed point spectra.   These maps naturally endow the homotopy groups of a $G$-spectrum with the structure of a Mackey functor, and Mackey functors are thus the key structure that arises when considering equivariant spectra in terms of their fixed points, as explained in \cite{alaska}.  Thus, a general construction of $G$-spectra in a way that extends a construction equivariant Eilenberg--MacLane spectra is a fundamental part of a complete understanding of $G$-spectra from the fixed-point perspective.  Our main theorem is designed to supply such a construction.

 In this paper, we give a machine for producing a genuine equivariant spectrum from an appropriate  ``Mackey functor'' of categorical data, defined in terms of an enriched functor out of a permutative category enriched version of the Burnside category.  Our main theorem can be loosely stated as follows. We make this precise in Section \ref{infiniteloops}.
\begin{thrm}[Main Theorem]\label{thrm:inftloopmachine}
Let $G$ be a finite group.  A ``Mackey functor of permutative categories'' determines a spectral Mackey functor $\BGspec^\op\to \Spec$, which by Guillou--May \cite{GM2011} represents a genuine $G$-spectrum.
\end{thrm}

Moreover, because the structure of fixed points is inherent in this construction, the Mackey functor of homotopy groups of the resulting spectrum is easily accessible.  In particular, this makes it easy to build Eilenberg--MacLane spectra for Mackey functors.
\begin{cor}\label{thrm:EMspectra}
Theorem \ref{thrm:inftloopmachine} provides a functorial construction of Eilenberg--MacLane spectra for Mackey functors.
\end{cor}

 Our construction is more than just a method for constructing Eilenberg--MacLane spectra: it also produces the sphere spectrum and other suspension spectra of finite $G$-sets.
This result is an equivariant analogue of the Barratt--Priddy--Quillen Theorem \cite{BPQthrm} which is an inherent part of any good construction of spectra from algebraic data. 

\begin{cor}\label{thrm:suspensionspectra}
Theorem \ref{thrm:inftloopmachine} produces suspension spectra for finite $G$-sets from input built solely from the category of finite $G$-sets. 
\end{cor}

 While other methods of constructing equivariant suspension spectra and Eilenberg--MacLane spectra exist, Theorem \ref{thrm:inftloopmachine} transparently unites these examples as special cases of the same general construction.

An essential tool in proving Theorem \ref{thrm:inftloopmachine} is a good nonequivariant $K$-theory machine for constructing spectra from permutative categories, which will be denoted $\K$.  We require our $K$-theory machine to be a multifunctor, as explained in Section \ref{sect:pairings}.  Multifunctoriality can be thought of as an extension of a theory of ``pairings'' for $\K$. Pairings originally arose as a way to control multiplicative-type structure on spectra.  A $K$-theory machine with pairings was first described in \cite{maypairings}, although not all of the details we require here appear there.  A complete account will appear in \cite{MMO}. An explicitly  multifunctorial $K$-theory machine is given by the multicategory version of Elmendorf and Mandell \cite{EM2006}.  However, to date there has been no comparison of the multiplicative or pairing structure for the two $K$-theory machines.  In order to apply the equivariant foundations of \cite{GM2011}, we choose to use their $K$-theory machine, which is an updated version of the one in \cite{maypairings}.  Nevertheless, the categorical constructions that follow  work equally well with either $K$-theory machine.  In particular, our primary technical construction is the following theorem, which can be proved using any sufficiently nice nonequivariant  multiplicative $K$-theory machine.
\begin{thrm}\label{thrm:constructk} There is a spectrally-enriched functor $\kobj\col \K_\bullet(\perm)\to \Spec$ which sends a permutative category $\A\in \perm$ to its $\K$-theory spectrum $\K(\A)$.
\end{thrm} 
Here $\K_\bullet(\perm)$ is a spectrally enriched category of permutative categories, which will be defined explicitly in Section \ref{constructk}.

To orient the reader familiar with this area, we relate Theorem \ref{thrm:constructk} to other $K$-theory constructions relating permutative categories and spectra.  Elmendorf and Mandell \cite{EM2006} construct a simplicially-enriched functor from the multicategory of permutative categories to the multicategory of symmetric spectra.  Guillou \cite{Guillou2010} discusses the construction of spectrally enriched versions of categories ``enriched in symmetric monoidal categories'' using the work of \cite{EM2006}; a similar construction using May's $K$-theory machine is given in \cite{GM2011}.  We expand and combine these results to produce our spectrally-enriched functor from the category of permutative categories to the category of spectra.

The first several sections of this paper are devoted to establishing the categorical constructions needed to define the input of our construction. We then carry out the construction  and discuss examples of its application.  In Section \ref{sect:background}, we remind the reader of the definitions and tools needed for enrichment in a multicategory.  This language provides the basic framework for our categorical constructions. In Section \ref{sect:pairings}, we discuss the multicategory of permutative categories and state our requirements for the $K$-theory machine. We are explicitly interested in categories enriched in the multicategory of permutative categories, which are discussed in greater detail in Section \ref{sect:PCcat}, and in Section \ref{sect:perm} we discuss our main example, which is a self-enrichment of the category of permutative categories. In Section \ref{constructk}, we prove our main technical result,  Theorem \ref{thrm:constructk}. Section \ref{infiniteloops} assembles this data to prove the main theorem, Theorem \ref{thrm:inftloopmachine}.  In Section \ref{EMspec}, we apply our main theorem  to produce an especially nice construction of Eilenberg--MacLane spectra for Mackey functors. Finally, in  Section \ref{sect:suspspec}, we use this construction to build  suspension spectra for finite $G$-sets, including the sphere spectrum.

\subsection*{Acknowledgments}
The authors would like to thank Clark Barwick for the suggestion that kicked off this project.  They thank Peter May and Bert Guillou for sharing working drafts. Thanks also to Peter May, Bert Guillou, Mike Hill, Mona Merling and Andrew Blumberg for many interesting conversations that contributed the development of this paper. They would also like to thank Kirsten Wickelgren for a very insightful question that helped clarify some of the key ideas. The authors would like to thank Kathryn Hess, Paul Goerss and John Greenlees for helpful comments on this work. Finally, the authors thank an anonymous referee for comments that improved the exposition.

\section{Notation and categorical background}\label{sect:background}

One of the primary structures we use to define the categories and functors used in Theorem \ref{thrm:inftloopmachine} is the structure of enrichment in a multicategory.  This notion generalizes the more familiar notion of enrichment in a monoidal category.  In this section we will recall the basic categorical definitions involved in multicategory enrichments  and establish the notation used in the rest of the paper.

We will first recall the definition of a multicategory. For a full account, see Leinster \cite[\S 2.1]{leinster}. 
\begin{defn}A \emph{multicategory} $\M$ consists of a collection of objects, denoted $ob\M$, and for each $k\geq 0$ and objects $a_1,\cdots a_k, b$, a set $\M(a_1,\dots, a_k; b)$ of \emph{$k$-morphisms} $(a_1,\dots, a_k)\to b$.  The morphism sets are related by composition functions
\[
\xymatrix{
\makebox[.2\textwidth]{$\M(b_1,\dots b_n;c)\times \M(a_1^1,\dots , a_1^{k_1};b_1)\times \cdots \times \M(a_n^1,\dots , a_n^{k_n};b_n)$}\ar[d]\\
 \M(a_1^1,\dots , a_1^{k_1}, \dots ,a_n^1,\dots , a_n^{k_n};c).
}
\]
These functions are denoted by $(g; f_1,\dots f_n)\mapsto g\circ (f_1,\dots, f_n)$. For each object $a$, there is an identity morphism $1_a\in \M(a;a)$. All of this data is subject to associativity and identity conditions, which can be found in \cite[\S 2.1]{leinster}.
\end{defn}

 When working diagrammatically, we depict composition in $\M$ as
\[(a_1^1,\dots , a_1^{k_1}, \dots ,a_n^1,\dots ,a_n^{k_n}) \xrightarrow{(f_1,\dots f_n)} (b_1,\dots b_n)\xrightarrow{g} c.\]
This allows us to express equalities among composite morphisms in terms of commutative diagrams in $\M$. 

There is also a notion of \emph{symmetric multicategory} in which there is an action of the symmetric group $\Sigma _k$ on the collection of all $k$-morphisms, compatible with the action on the inputs and with composition.  In particular, there is a twist action on 2-morphisms. All of our examples are in fact symmetric.

\begin{ex}The main example for this paper will be $\permmc$, the multicategory of permutative categories and multilinear maps. This multicategory is the subject of Section~\ref{sect:pairings}.
\end{ex}

\begin{ex}\label{monoidalismulti} A monoidal category $(\V,\otimes)$ has an underlying multicategory structure.  The objects of the multicategory are the objects of $\V$, and the $k$-morphisms are given by morphisms $a_1\otimes \cdots \otimes a_k\to b$ in $\V$. If $\V$ is symmetric as a monoidal category, it is also symmetric as a multicategory. Many of the multicategories used in this paper arise in this way.  In this case, we use the same notation for the monoidal category and the multicategory.
\end{ex}

Given multicategories $\M$ and $\N$, a \emph{multifunctor} $F\colon \M \to \N$ consists of an assignment of objects $F\colon ob\M \to ob\N$, and for objects $a_1,\dots,a_k,b$ of $\M$, a function on morphism sets
\[ F\colon \M(a_1,\dots,a_k;b)\to \N(F(a_1),\dots,F(a_k);F(b)),\]
compatible with identity and composition.

\begin{ex}If $\C$ and $\D$ are monoidal categories, a  lax monoidal functor $F\colon \C \to \D$ canonically gives rise to a multifunctor of the corresponding multicategories. 
\end{ex}

We next define the notion of enrichment in a multicategory. Let $\M$ be a multicategory. 

\begin{defn}An \emph{$\M$-enriched category} $\C$ consists of a collection of objects, $ob\C$, and for objects $a,b$ in $\C$,  a morphism object  $\C(a,b)$ which is an object of $\M$. For objects $a,b,c$ of $\C$, composition is given by a 2-morphism in $\M$,
\[\comp \colon (\C(b,c),\C(a,b))\to \C(a,c),\]
and the identity is given by a 0-morphism in $\M$ with the empty sequence as source,
\[\id_a\colon ()\to \C(a,a).\]
These morphisms are subject to the usual associativity and identity constraints, which are expressed as follows. Given objects $a,b,c,d$ in $\C$, the following diagram in $\M$ commutes:
\begin{equation}\label{eqn:assoc}
\xymatrixcolsep{2cm}\xymatrix{
(\C(c,d),\C(b,c),\C(a,b)) \ar[r]^-{(\comp,1_{\C(a,b)})} \ar[d]_{(1_{\C(c,d)},\comp)} & (\C(b,d),\C(a,b))\ar[d]^{\comp}\\
(\C(c,d),\C(a,c)) \ar[r]_-{\comp} & \C(a,d).
}
\end{equation}
For objects $a$ and $b$, the diagram
\begin{equation}\label{eqn:identity}
\xymatrixcolsep{2cm}\xymatrixrowsep{1.5cm}\xymatrix{
(\C(a,b))\ar[r]^-{(\id_b,1_{\C(a,b)})} \ar[d]_{(1_{\C(a,b)},\id_a)} \ar[rd]^-{1_{\C(a,b)}} & (\C(b,b),\C(a,b))\ar[d]^{\comp}\\
(\C(a,b),\C(a,a)) \ar[r]_-{\comp} & \C(a,b)
}
\end{equation}
commutes.
\end{defn}

\begin{rem} If the enriching multicategory $\M$ is symmetric, any $\M$-enriched category $\C$ has a well-defined opposite category $\C^\op$.  The objects of $\C^\op$ are the same as the objects of $\C$; morphism objects in $\C^\op$ are defined by setting $\C^\op(a,b)$ equal to $\C(b,a)$.  The definition of composition in $\C^\op$ requires use of the twist action on 2-morphisms in $\M$.  This is a generalization of what happens in the case of enrichments in monoidal categories: here again, the existence  of opposite categories follows from  a symmetric monoidal structure on the enriching category. 
\end{rem}

\begin{defn}\label{defn:enrfunctor}
Let $\C$ and $\D$ be $\M$-enriched categories. An \emph{$\M$-enriched functor} $F\colon \C \to \D$ consists of an assignment $F\colon ob\C \to ob\D$, and for each pair of objects $a$ and $b$ in $\C$, a morphism in $\M$,
\[F\colon \C(a,b) \to \D(Fa,Fb),\]
compatible with composition and identity, as expressed by the commutativity of the diagrams
\[
\xymatrix{
(\C(b,c),\C(a,b))\ar[r]^-{\comp} \ar[d]_{(F,F)} & \C(a,c) \ar[d]^{F}\\
(\D(Fb,Fc),\D(Fa,Fb)) \ar[r]_-{\comp} & \D(Fa,Fc),
}
\]
and
\[
\xymatrix{
() \ar[r]^-{\id_a}\ar[dr]_-{\id_{Fa}}&  \C(a,a) \ar[d]^{F}\\
& \D(Fa,Fa).
}
\]
\end{defn}

We can similarly define $\M$-enriched natural transformations. The collection of $\M$-enriched categories, $\M$-enriched functors and $\M$-enriched natural transformations forms a 2-category, denoted $\M$-$\cat$. 
\begin{rem}
Note that for a monoidal category $\V$, the notions of enrichment over $\V$ as a monoidal category and as a multicategory coincide, so that the term $\V$-$\cat$ is unambiguous.
\end{rem}

One of the basic constructions in the sequel is changing enrichments along a multifunctor. 
\begin{prop}\label{prop:basechange} A multifunctor $F\colon \M \to \N$ induces a 2-functor
\[F_{\bullet} \colon \M\text{-}\cat \to \N\text{-}\cat,\]
which sends an $\M$-enriched category $\C$ to the $\N$-enriched category $F_\bullet\C$ with the same collection of objects as $\C$, and for which  $(F_{\bullet}\C)(a,b)$ is given by $F(\C(a,b))$.  Again, this generalizes the situation of enrichment over a monoidal category.
\end{prop}

\section{The multicategory of permutative categories}\label{sect:pairings}

In this section we introduce the multicategory of permutative categories. Recall that a permutative category is a symmetric monoidal category for which the associativity and unit constraints are the identity. For all permutative categories we will denote the monoidal product by $\oplus$, the unit element by 0, and the symmetry isomorphism by $\gamma$. We assume permutative categories, functors between them, and natural transformations are topologically enriched.

The definitions in this section are taken from  those in \cite[\S 3]{EM2006}. The definitions are also closely related to \cite[Definition 10]{HP}. 

\begin{defn}\label{defn:linear}  Let $\A$ and $\B$ be permutative categories. A \emph{strictly unital lax symmetric monoidal functor} $(f,\delta)\colon \A\to\B$ consists of a functor $f\colon \A\to\B$ of underlying categories, together with a  natural transformation
 \[
  \delta \colon   f(a)\oplus f(a')\to f(a\oplus a'),
 \]
called the \emph{structure morphism}.  This data must  satisfy the following conditions: 
 \begin{enumerate}
  \item $f(0)=0$;
  \item if either $a$ or $a'$ is 0, then $\delta$ is the identity morphism;
  \item for all $a,a',a'' \in \A$ , the following diagram commutes:
  \[
   \xymatrix{
 f(a)\oplus f(a') \oplus f(a'')  \ar[r]^-{\id \oplus \delta}  \ar[d]_{\delta \oplus \id} &  f(a)\oplus f(a'\oplus a'')  \ar[d]^{\delta}\\
   f(a\oplus a') \oplus f(a'') \ar[r]_-{\delta} & f(a\oplus a' \oplus a''); 
   }
  \]
  \item for all $a,a'\in \A$, the following diagram commutes:
  \[
   \xymatrix{
f(a)\oplus f(a')\ar[r]^-{\delta} \ar[d]_{\gamma} & f(a\oplus a') \ar[d]^{f(\gamma)}\\
f(a') \oplus f(a) \ar[r]_-{\delta}  & f(a' \oplus a).
   }
  \]
 \end{enumerate}
\end{defn}

\begin{rem}
The structure morphism $\delta$ should be thought of as requiring a ``linearity'' condition on $f$: the conditions above generalize the definition of a linear map between abelian groups to the context of symmetric monoidal categories.  Requiring our functors to be strictly unital is analogous to working with based maps in the context of spaces. This is not a serious restriction; any lax symmetric monoidal functor that preserves the unit up to isomorphism is suitably equivalent to one that is strictly unital.
\end{rem}

\begin{defn}\label{defn:multilinear}
Let $\A_1,\dots, \A_k$ and $\B$ be permutative categories. A \emph{multilinear functor} $(f;\delta_1,\dots \delta_k)\colon (\A_1,\dots ,\A_k)\to\B$ consists of a functor $f\colon \A_1\times \cdots \times \A_k\to\B$ of underlying categories, together with  natural transformations
\[\delta_i \colon f(a_1,\dots, a_i, \dots, a_k) \oplus  f(a_1,\dots, a'_i, \dots, a_k) \to  f(a_1,\dots, a_i \oplus a'_i, \dots, a_k)\]
called  \emph{linearity constraints}.  These constraints must satisfy conditions given in  \cite[Definition 3.2]{EM2006}. In particular, $f(a_1,\dotsc, a_k)=0$ if any $a_j=0$.  Additionally, the transformation $\delta_i$ is the identity map if either $a_i$, $a_i'$ or any other $a_j$ is $0$.
\end{defn}

The natural transformations $\delta_i$ should be thought of as saying that the functor $f$ is linear in each variable, up to a constraint given by $\delta_i$. The required conditions encode an associativity of this constraint and a compatibility with the symmetry.  These are similar to conditions (3) and (4) of Definition \ref{defn:linear}. There is one further condition that relates the constraints for different variables. We will denote the multilinear functor $(f; \delta_1, \dots, \delta_n)$ simply by $f$ when there is no risk of confusion.

\begin{rem} When $k=1$, Definition \ref{defn:multilinear} coincides with Definition \ref{defn:linear}. For $k=2$, Definition \ref{defn:multilinear} is a slight generalization of the notion of pairing defined in \cite{maypairings} and used implicitly in \cite{Guillou2010}. 
\end{rem}

\begin{defn} Let $\permmc$ be the multicategory whose objects are small permutative categories and whose multimorphisms are given by multilinear functors between them.   A  0-morphism $a\colon ()\to \A$ is given by a choice of object $a$ in $\A$.  Given multilinear functors $(g;\{\delta _j^g\}_{j=1}^n)\colon (\B_1,\dots, \B_n) \to \C$ and $(f_j;\{\delta ^{f_j}_i\}_{i=1}^{k_j})\colon (\A_j^1,\dots ,\A_j^{k_j})\to B_j$ for $j=1,\dots ,n$, we define their composite by
\[
(g;\{\delta _j^g\}_{j=1}^n)\circ \left((f_1;\{\delta ^{f_1}_i\}_{i=1}^{k_1}),\dots, (f_n;\{\delta ^{f_n}_i\}_{i=1}^{k_n})\right)=(g\circ(f_1\times \cdots \times f_n); (\delta_s)_{s=1}^{\sum k_j})
\]
where $\delta_s$ is given by the appropriate composite of $\delta_j ^g$ with $g(\delta _i ^{f_j})$. See \cite[Definition 3.2]{EM2006} for details.
\end{defn}
This definition of composition generalizes that of composition of lax  monoidal functors, which we recall here to orient the reader. Let $(g;\delta ^g)\colon \B \to \C$ and $(f; \delta ^f)\colon \A \to \B$ be strictly unital lax  monoidal functors. Their composite is given by $(g\circ f; \delta)$, where $\delta$ is the natural transformation given by the composite
\[ g(f(a))\oplus g(f(a')) \xrightarrow{\delta ^g} g(f(a)\oplus f(a')) \xrightarrow{g(\delta ^f)} g(f(a\oplus a')).\]

Our work requires a $K$-theory machine that respects the multicategory structure on permutative categories.  The following theorem makes this requirement precise.
 \begin{thm}[\textit{cf}.\ {\cite[Theorem 1.6, Theorem 2.1]{maypairings}, \cite[Theorem 1.1]{EM2006}}]\label{Kpairs} 
There is a  $K$-theory functor $\K$ from permutative categories to spectra has the structure of a multifunctor
\[\K \colon \permmc \to \Spec. \]
\end{thm}
Here we are viewing the monoidal category of spectra under smash product as a multicategory as in Example \ref{monoidalismulti}.

We will explicitly use two important consequences of the multifunctoriality of $\K$.
\begin{cor}\label{cor:Kofmultimap}
A multilinear functor of permutative categories \[f\colon (\A_1,\dots ,\A_k)\to\B \] induces a map of spectra
\[ \K(\A_1)\sma \cdots \sma \K(\A_k)\to \K(\B),\]
which we denote by $\K(f)$. Furthermore, this assignment is compatible with composition of multilinear maps.  
\end{cor}
 
Recall that a 0-morphism $b\colon ()\to \B$ in $\permmc$ is given by a choice of an object $b\in B$.  The inclusion of the object $b$ into $\B$ extends to a strict symmetric monoidal functor
\[i_b\colon \mathcal{F}\to \B\]
where $\mathcal{F}$ is the free permutative category on a single object. The category $\mathcal{F}$ is a skeletal version of the category of finite sets and isomorphisms, and the Barratt--Priddy--Quillen theorem thus identifies $\K(\mathcal{F})$ as the sphere spectrum $S$. The following lemma follows from the proof of Theorem \ref{Kpairs} that is given by Elmendorf--Mandell \cite{EM2006}.

\begin{lemma}\label{lemma:Kon0morph}
Let $b\colon ()\to \B$ be a 0-morphism in $\permmc$.  Then $\K(b)\colon S\to \K(\B)$ is the 0-morphism in $\Spec$ given by $\K(i_b)\colon S\to \K(\B)$
\end{lemma}

\begin{rem}
Theorem \ref{Kpairs} originates in \cite{maypairings}. However, the exposition in \cite{maypairings} is tailored for understanding the multiplicative structures on ring spectra, rather than the more general multilinear maps we use here. Together with modern advances in constructing categories of spectra with good monoidal structures, this makes \cite{maypairings} somewhat difficult to read.  As mentioned in the introduction, Elmendorf and Mandell \cite{EM2006} also produce a $K$-theory machine with the multifunctoriality of Theorem \ref{Kpairs}---their work is the origin of the multifunctor perspective on the subject.  Currently there is no comparison between the multiplicative-type structure produced by \cite{EM2006} and \cite{maypairings}; since we rely on the main theorem of \cite{GM2011}, which is proved using the pairings of \cite{maypairings}, we must also use the theory of \cite{maypairings}.  This theory has been modernized and expanded in the forthcoming \cite{MMO}, which will provide a more accessible resource, but the fundamental constructions are those of \cite{maypairings}.
\end{rem}

\section{$\PC$-categories}\label{sect:PCcat}

Our construction makes heavy use of categories enriched in the multicategory $\permmc$, which we call $\PC$-categories after Guillou \cite{Guillou2010}.  In this section, we describe their structure in a more explicit way.

\begin{defn} A $\PC$-\emph{category} is a category enriched in the multicategory $\permmc$. Similarly, a $\PC$-\emph{functor} is an enriched functor, and a $\PC$-\emph{natural transformation} is an enriched natural transformation.
\end{defn}

We now unpack some of these definitions for the reader. A $\PC$-category  consists of the following data:
\begin{enumerate}
\item a collection of objects, denoted by $ob\pcC$;
\item for each pair of objects $X,Y$, a permutative category $\pcC(X,Y)$;
\item for every triple of objects $X,Y,Z$, a bilinear map 
\[ \circ: (\pcC(Y,Z),\pcC(X,Y))\to \pcC(X,Z);\]
\item for every object $X$ in $ob\pcC$, an object $\id_X$ of $\pcC(X,X)$;
\end{enumerate}
such that Diagrams (\ref{eqn:assoc}) and (\ref{eqn:identity}) commute.

Note that if we forget the permutative structure on the categories of morphisms,  a $\PC$-category has an underlying 2-category.

\begin{rem} Our definition of a $\PC$-category is slightly more general than the one in \cite[\S 3]{Guillou2010}, in which one of the distributivity isomorphisms of the bilinear map $\circ$ in (3) is chosen to be the identity. By \cite[Theorem 1.2]{Guillou2010}, any $\PC$-category in our sense is appropriately equivalent to a $\PC$-category in Guillou's sense.
\end{rem}

Composition in $\pcC$ gives rise to precomposition and postcomposition maps. More explicitly, given objects $X$, $Y$ and $Z$ in the $\PC$-category $\pcC$, and an object $f$ of the permutative category $\pcC(X,Y)$, there exist strictly unital lax symmetric monoidal functors
\[ f^{\ast} \colon \pcC(Y,Z)\to \pcC(X,Z) \quad \text{and} \quad f_{\ast} \colon \pcC(Z,X) \to \pcC(Z,Y),\]
given by precomposition and postcomposition by $f$ respectively. The monoidal constraints are defined by using the distributivity constraints of $\circ$.

 A $\PC$-functor $F\colon \pcC \to \pcD$ consists of an assignment of objects $F \colon ob \pcC \to ob\pcD$ together with strictly unital lax symmetric monoidal functors
\[F_{X,Y}\colon \pcC (X,Y) \to \pcD (FX, FY)
\]
such that the diagrams in Definition \ref{defn:enrfunctor} commute.

Given  $\PC$-functors $F$ and $G$ from $\pcC$ to $\pcD$, a $\PC$-natural transformation $\alpha\colon F \to G$ consists of a collection of morphisms in $\pcD$, 
\[ \alpha _X \colon F(X) \to G(X),\]
such that the following diagram in $\permmc$ commutes
\[
\xymatrix{
\pcC(X,Y) \ar[r]^{G} \ar[d]_{F} & \pcD(GX,GY)\ar[d]^{(\alpha _X)^{\ast}}\\
\pcD(FX,FY)\ar[r]_{(\alpha_Y)_{\ast}} & \pcD(FX,GY).
}
\]

The following theorem follows directly from Proposition \ref{prop:basechange} and Theorem \ref{Kpairs}.

 \begin{thrm}[\textit{cf}.\ {\cite[Theorem 1.1]{Guillou2010}}] \label{KthryofPCcat} 
There is a 2-functor
\[
\K _{\bullet} \colon \PC\textrm{-}Cat \to \Spec\textrm{-}Cat\]
that sends a $\PC$-category $\pcC$ to a spectrally enriched category with the same collection of objects and with morphism spectra from $X$ to $Y$ given by $\K \pcC(X,Y)$.
\end{thrm}

\section{The $\PC$-category $\perm$}\label{sect:perm}
Let $\perm$ denote the 2-category of small permutative categories, strictly unital lax symmetric monoidal functors and monoidal natural transformations. As noted by Guillou \cite[Example 3.4]{Guillou2010}, this 2-category has the structure of a $\PC$-category.  In particular, the morphism category $\perm(\A,\B)$ between any two permutative categories is itself a permutative category. 

Concretely, the monoidal structure on $\perm(\A,\B)$ is given by pointwise addition so that $(f\oplus g)(a)$ is defined to be $f(a)\oplus g(a)$.  The structure morphism for $f\oplus g$  is given by the composite
\[ f(a)\oplus g(a) \oplus f(a') \oplus g(a') \xrightarrow{1\oplus \gamma \oplus 1} f(a)\oplus f(a') \oplus g(a) \oplus g(a') \xrightarrow{\delta^f \oplus \delta^g} f(a\oplus a') \oplus g(a\oplus a').\]
One easily checks that this gives $\perm(\A,\B)$ the structure of a permutative category.

Given permutative categories $\A$, $\B$ and $\C$, the usual composition functor extends to a map
\[ \circ \colon \left(\perm(\B, \C), \perm(\A, \B)\right) \to \perm(\A, \C)\]
which is bilinear in the sense of Definition \ref{defn:multilinear}. 
The first linearity constraint $\delta_1$ is the identity, since it follows from the definitions that $(g\circ f) \oplus (g'\circ f)=(g\oplus g')\circ f$. The second linearity constraint
\[ \delta_2\colon  (g\circ f) \oplus (g \circ f') \to g\circ(f\oplus f')\]
is given by the structure morphism of $g$. It is routine to check that the necessary diagrams commute.

\begin{rem}
We require the objects in $\perm(\A,\B)$ to be strictly unital functors so that the composition bilinear map is also strictly unital, which is required in Definition \ref{defn:multilinear}. The reason for this requirement is that the available $K$-theory machines are only multifunctorial with respect to strictly unital multilinear maps. 
\end{rem}

The $\PC$-category $\perm$ has a bilinear evaluation map. This structure is important in the construction of the functor of Theorem \ref{thrm:constructk}.

\begin{prop}\label{prop:evpairing}Given permutative categories $\A$ and $\B$, there exists an evaluation map 
\[ev\colon (\perm(\A,\B), \A) \to \B\]
which is a bilinear functor in the sense of Definition \ref{defn:multilinear}.
\end{prop}
The underlying functor is defined in the usual way on objects and morphisms, so that $ev(f,a)=f(a)$; the key point is to note that strict unitality and the natural transformations  witnessing the lax monoidality of $f$ determine a bilinear structure on $ev$.
More explicitly, the first linearity constraint is the identity, since the equality
\[\delta_1\colon ev(f,a)\oplus ev(f',a)=ev(f\oplus f',a)\]
follows immediately from the definition of $f\oplus f'$, and the second linearity constraint,
\[ \delta_2\colon ev(f,a)\oplus ev(f,a')\to ev(f,a\oplus a'),\] 
is given by the structure morphism of $f$. Similarly, $ev(0,a)=0$ and $ev(f,0)=0$; note the second equality expresses the strict unitality of $f$. 
We leave it to the reader to check that the appropriate diagrams commute.

These evaluations are compatible, in the sense that the following diagram commutes in the multicategory $\permmc$:
\begin{equation}\label{trilineareval}
\xymatrix{(\perm(\B,\C), \perm(\A,\B), \A) \ar[d]_{(\id, ev)}\ar[r]^-{(\circ, \id)}  & (\perm(\A,\C),\A) \ar[d]^{ev}\\
(\perm(\B,\C)\, \B) \ar[r]_-{ev} & \C}
\end{equation}
That this diagram commutes follows from the underlying definition of evaluation; it is straightforward to check that the required diagrams for compatibility of the linearity transformations commute.

Furthermore, evaluation satisfies the following universal property. 
\begin{lemma}\label{lemma:closed} Let $\A$, $\B$ and $\C$ be permutative categories. Let $(f, \delta _1, \delta_2) \colon (\A, \B) \to \C$ be a bilinear map. Then there exists a unique strictly unital lax monoidal functor
\[g\colon A \to \perm(\B,\C) \]
such that the following diagram commutes
\[
\xymatrix @C=4pc{
(\A,\B) \ar[r]^-{(g,1_{\B})} \ar[dr]_{f}&(\perm(\B,\C),\B) \ar[d]^{ev}\\
&\C.
}
\]
\end{lemma}

\begin{proof} On objects, we let $g$ send $a$ in $\A$ to the strictly unital lax monoidal functor given by the composite
\[ \B \fto{(a,1_\B)} (\A,\B) \fto{f} \C.\]
In particular, as an underlying functor, $g(a)(b)=f(a,b)$. A morphism $l\colon a\to a'$ in $\A$ determines a monoidal natural transformation
\[ g(l)\colon g(a)\to g(a')\]
whose component at $b\in \B$ is given by $f(l,\id_b)$. 
It is routine to check that $g$ is strictly unital and lax monoidal, with constraint given by $\delta_1$. It is also easy to check that such a $g$ is unique. 
\end{proof}

\begin{rem}\label{rem:otherwaystoseeclosedmonoidalstuff} Heuristically, Proposition \ref{prop:evpairing} and Lemma \ref{lemma:closed} should be thought of as part of a closed structure on the multicategory $\perm$. This idea is made precise in \cite{closedmulticat}.  The closed structure may also be viewed as arising from an embedding of the multicategory $\perm$ into the closed symmetric monoidal category of ``based multicategories'' which is discussed by Elmendorf--Mandell in \cite{EM2009}.  In order to make this paper more self-contained, we chose to give explicit constructions of the structures needed for our results. 
\end{rem}

\section{A spectrally enriched functor $\kobj\colon\K_\bullet(\perm)\to \Spec$}\label{constructk}

Let $\K_\bullet(\perm)$ denote the spectral category constructed from $\perm$ by Theorem \ref{KthryofPCcat}.  That is, $\K_\bullet(\perm)$ is the category enriched over spectra whose objects are the same as the objects of $\perm$, and for which given permutative categories $\A$ and $\B$, the spectrum of morphisms is given by the $K$-theory spectrum $\K_\bullet(\perm(\A,\B))$.  

We define a spectrally enriched functor $\kobj\colon\K_\bullet(\perm)\to\Spec$, thus proving Theorem \ref{thrm:constructk}. For a permutative category $\A$, $\kobj(\A)$ is the spectrum $\K\A$. The map of spectra
\[
 \kobj\colon \K(\perm(\A,\B))\to F(\K\A,\K\B)
\]
is defined as the adjoint of the map
\[
 \K(ev)\colon \K(\perm(\A,\B))\sma \K\A \to \K\B,
\]
whose existence follows from Proposition \ref{prop:evpairing} and  multilinearity of $\K$ as in Corollary \ref{cor:Kofmultimap} of Theorem \ref{Kpairs}.

By definition of the adjoint we thus have that the following diagram commutes.
 \begin{equation}\label{adjointdiagram}
  \xymatrix{
  \K(\perm(\A,\B))\sma \K\A \ar[d]_{\kobj\sma \id} \ar[rd]^-{\K(ev)}\\
  F(\K\A,\K\B)\sma \K\A \ar[r]_-{ev} & \K\B.
  }
 \end{equation}

\begin{thm} 
The map $\kobj$ defines a spectrally-enriched functor. 
\end{thm}

\begin{proof} We must check that our definition of $\kobj$ respects composition and units.  We turn first to composition. We must show that the following diagram commutes:
 \[
  \xymatrix{
  \K(\perm(\B,\C))\sma \K(\perm(\A,\B)) \ar[r]^-{\K(\circ)} \ar[d]_{\kobj\sma \kobj} & \K(\perm(\A,\C))\ar[d]^{\kobj}\\
  F(\K\B,\K\C)\sma F(\K\A,\K\B) \ar[r]_-{\circ} & F(\K\A,\K\C).
  }
 \]
Note that the existence of the top map $\K(\circ)$  follows from the fact that composition in a $\PC$-category is a bilinear map.

We show this diagram commutes by proving commutativity of its adjoint diagram
\begin{equation}\label{theadjointdiagram}
 \xymatrix{
  \K(\perm(\B,\C))\sma \K(\perm(\A,\B))\sma \K\A \ar[r]^-{\K(\circ)\sma \id} \ar[d]_{\kobj\sma \kobj\sma id} & \K(\perm(\A,\C))\sma \K\A\ar[d]^{\K(ev)}\\
  F(\K\B,\K\C)\sma F(\K\A,\K\B)\sma \K\A \ar[r]_-{\widetilde{\circ}} & \K\C,
  }
\end{equation}
where $\widetilde{\circ}$ denotes the adjoint of $\circ$, that is, $\widetilde{\circ}$ is defined as the composite
\[
 F(\K\B,\K\C)\sma F(\K\A,\K\B) \sma \K\A \fto{\circ\sma \id} F(\K\A,\K\C)\sma \K\A \fto{ev} \K\C.
\]
Diagram (\ref{theadjointdiagram}) decomposes into the following pieces.

\[
\def\objectstyle{\scriptstyle}
\def\labelstyle{\scriptstyle}
\xy
(0,0)*+{\K(\perm(\B,\C))\sma \K(\perm(\A,\B))\sma \K\A}="A"; (100,0)*+{\K(\perm(\A,\C))\sma \K\A}="B"; (0,-40)*+{F(\K\B,\K\C)\sma F(\K\A,\K\B)\sma \K\A}="F"; (75,-20)*+{\K(\perm(\B,\C))\sma \K\B}="D"; (20,-20)*+{\K(\perm(\B,\C))\sma F(\K\A,\K\B)\sma \K\A} ="C"; (50,-40)*+{F(\K\B,\K\C)\sma \K\B}="E"; (100,-40)*+{\K\C}="G"; (50,-60)*+{F(\K\A,\K\C)\sma \K\A}="H";
{\ar^-{\K(\circ)\sma \id} "A";"B" };
{\ar@/_2pc/_-{\kobj\sma \kobj\sma \id} "A";"F"};
{\ar^-{\id \sma \K(ev)} "A"; "D"};
{\ar^*-<.4em>{\id\sma \kobj\sma \id} "A"; "C"};
{\ar^-{\K(ev)} "B"; "G"};
{\ar^(.3){\kobj\sma \id\sma \id} "C"; "F"};
{\ar^-{\id\sma ev} "C"; "D"};
{\ar_-{\kobj\sma \id} "D"; "E"};
{\ar^-{\K(ev)} "D"; "G"};
{\ar^-{ev} "E"; "G"};
{\ar^-{\id\sma ev} "F"; "E"};
{\ar^-{\circ\sma \id} "F"; "H"};
{\ar^-{ev} "H"; "G"};
(3,-11)*+[o][F-]{1};
(27,-11)*+[o][F-]{2};
(80,-11)*+[o][F-]{3};
(35,-30)*+[o][F-]{4};
(75,-30)*+[o][F-]{5};
(50,-50)*+[o][F-]{6};
\endxy
\]

Subdiagrams 1 and 4 commute by the functoriality of the smash product.  Subdiagrams 2 and 5 commute because they are instances of Diagram (\ref{adjointdiagram}), which commutes. Subdiagram 3 comes from applying $\K$ to the commuting Diagram (\ref{trilineareval})  displaying the compatible multilinear structure of composition and thus commutes.  Finally, the commutativity of Subdiagram 6 follows from standard properties of closed monoidal categories.

Next we show that $\kobj$ respects the units in the spectral categories $\K_\bullet(\perm)$ and $\Spec$.  We must show that the diagram
\[  \xymatrix{
  S \ar[r] \ar[rd] &\K\perm(\A,\A)\ar[d]^{\kobj}\\
  & F(\K\A,\K\A)
  }
\]
commutes, where the top map is the map given by applying Lemma \ref{lemma:Kon0morph} to the object $\id_\A$ in $\perm(\A,\A)$ and the diagonal map is the usual unit map for $\Spec.$  This diagram commutes because its adjoint diagram
\begin{equation}\label{unitadj}  \xymatrix{
  S\sma \K\A \ar[r] \ar[rd] &\K\perm(\A,\A)\sma \K\A \ar[d]^{\K(ev)}\\
  & \K\A
  }
\end{equation}
commutes.  To see this, consider the left diagonal map $S\sma\K\A\to \K\A$, which is the unit isomorphism in $\Spec$.  

Recall that $S=\K\finset$ and that this unit morphism is given by applying $\K$ to the bilinear map
\[(\finset,\A)\to \A\]
which is determined by sending $(*,a)\in \finset\times \A$ to $a\in \A$, where $*$ is the generating object of $\finset$. We obtain Diagram (\ref{unitadj}) by applying $\K$ to the following commutative diagram  in $\permmc$:
\[
\xymatrix@C=4em{(\finset,\A)\ar[r]^-{(i_{1_{\A}},1_{\A})}\ar[dr] &(\perm(\A,\A),\A)\ar[d]^{ev}\\ 
&\A,}
\]
Here $i_{1_{\A}}\colon \finset\to \perm(\A,\A)$ is the strictly unital lax  monoidal functor induced by sending $*\in\finset$ to $1_{\A}\in \perm(\A,\A)$, as is used in Lemma \ref{lemma:Kon0morph}.
\end{proof}

In order to better understand how the functor $\kobj$ works, we now analyze what it does after passing to $\pi_0$.

Since $\pi_0\colon \Spec\to \ab$ is a lax monoidal functor, by Proposition \ref{prop:basechange} it induces a change-of-enrichments functor 
\[(\pi_0)_\bullet\colon \Spec\textrm{-}Cat \to \ab\textrm{-}Cat \]
which is given by applying $\pi_0$ at the level of morphisms.  Thus we obtain an additive functor $(\pi_0)_\bullet\kobj\colon (\pi_0)_\bullet\K_\bullet\perm\to (\pi_0)_\bullet\Spec$ from our spectral functor $\kobj$.  On objects, the functor $(\pi_0)_\bullet\kobj$ is the same as $\kobj$---it sends a permutative category $\A$ to the spectrum $\K\A$.  On morphisms, the behavior of $(\pi_0)_\bullet$ is summarized in the following proposition.

Recall that for any permutative category $\mathcal{P}$, it is a fundamental property of $K$-theory that $\pi_0\K\mathcal{P}$ is the group completion of the abelian monoid of connected components of objects of $\mathcal{P}$.  Thus the morphism group $(\pi_0)_\bullet\K_\bullet\perm(\A,\B)$ is the group completion of the set connected components of strictly unital lax monoidal functors $f\colon \A\to \B$.
As such, a map out of this group is determined by where it sends such functors.

\begin{prop}\label{pi0worksspectralevel}
Let $\A$ and $\B$ be permutative categories.  The map on morphism  groups
\[ \pi_0\K \perm(\A,\B) \to \pi_0 F(\K\A,\K\B)\]
given by $(\pi_0)_\bullet\kobj$ sends the connected component of a strictly unital lax monoidal functor $f\colon \A\to \B$ to the homotopy class $[\K f]$, where $\K f$ is the image of $f$ under the $K$-theory functor $\K$.
\end{prop}

\begin{proof}
We must show that the following diagram commutes in the homotopy category of spectra: 
\[\xymatrix@C=3pc{
S\ar[r]^-{\K(i_f)}\ar[dr]_-{\widetilde{\K (f)}} & \K\perm(\A,\B)\ar[d]^-{\kobj}\\
&F(\K\A,\K\B).} 
\]
Here $\K(i_f)$ is the 0-morphism in $\Spec$ associated to the object $f\in \perm(\A,\B)$ in Lemma \ref{lemma:Kon0morph} and $\widetilde{\K f}$ is the morphism $S\to F(\K\A,\K\B)$ that is adjoint to the map $\K f\colon \K\A\to \K\B$.  Thus the desired commutative diagram has adjoint
\[
\xymatrix@R=1ex@!C@C=0em{
  S\sma \K\A \ar[rr]^-{\K(i_f)\sma \id} \ar[rd]_{\cong} &&\K\perm(\A,\B)\sma \K\A \ar[dd]^{\K(ev)}\\
  &\K\A \ar[rd]_{\K f}\\
  && \K\B.
  }
\]
The isomorphism $S\sma \K\A \to \K\A$ is the adjoint of the unit map and is displayed in Diagram (\ref{unitadj}).  Thus, recalling that $S=\K\finset$, this diagram is obtained by applying $\K$ to the following diagram:
\[\xymatrix@R=1ex{
(\finset,\A)\ar[rrr]^-{(i_f,\id)}\ar[dr]_-{(i_\id,\id)} &&& (\perm(\A,\B),\A)\ar[dd]^-{ev}\\
&(\perm(\A,\A),\A)\ar[dr]_-{ev}\\
&&\A\ar[r]^f & \B
}
\]
The left diagonal composite sends $(\ast,a)\in \finset\times\A$ to $a\in\A$.  It follows from the definition of evaluation that this is a commutative diagram of bilinear maps in $\permmc$.
\end{proof}

As a consequence of this proposition, we obtain  the following.
\begin{cor}\label{pi0works}
The functor $\kobj$ induces a commutative diagram of abelian groups
\[\xymatrix{ \pi_0\K\perm(\A,\B) \ar[r]\ar[dr]& \pi_0F(\K\A,\K\B) \ar[d]\\
& \ab(\pi_0\K\A,\pi_0\K\B)
}
\]
where the lower left map is determined by sending the class of strictly unital lax monoidal functor $f$ to the induced map on group completions arising from the restriction of $f$ to 
object sets.
\end{cor}

\section{Constructing genuine equivariant spectra}\label{infiniteloops}

The construction of the functor $\kobj$ of Section \ref{constructk}, when combined with the recent work of Guillou--May, yields a new method for constructing genuine $G$-spectra for finite groups of equivariance.  We will apply this method in the specific case of Eilenberg--MacLane spectra in Section \ref{EMspec} and suspension spectra of finite $G$-sets in Section \ref{sect:suspspec}.  The basic idea behind \cite{GM2011} is that genuine $G$-spectra are given by spectrally-enriched presheaves of spectra on a particularly accessible spectrally-enriched version of the Burnside category.  Our functor $\kobj$ allows us to construct such presheaves from the algebraic data of a $\PC$-functor from a $\PC$-category version of the Burnside category to $\perm$.

The $\PC$-category in question is defined using an explicit small version of the category of $G$-sets. We regard a $G$-set as a finite set of the form $(\mathbf{n},\alpha)$, where $\mathbf{n}=\{1,\dotsc,n\}$ and $\alpha$ is a homomorphism $G\fto{\alpha}\Sigma_n$ endowing $\mathbf{n}$ with a $G$-action.  A $G$-map $(\mathbf{n},\alpha)\to (\mathbf{m},\beta)$ is a function $f\col \mathbf{n}\to \mathbf{m}$ such that $f\circ\alpha(g)=\beta(g)\circ f$ for all $g\in G$.  The disjoint union of $A=(\mathbf{n},\alpha)$ and $B=(\mathbf{m}, \beta)$ is given by $A\amalg B=(\mathbf{n+m},\alpha+\beta)$ where $\alpha+\beta$ indicates the usual block sum homomorphism $G\to \Sigma_{n+m}$.  This model of disjoint union is strictly associative and unital.

\begin{defn}The \emph{Burnside category} $\B_G$ is a category whose objects are finite $G$-sets and whose morphisms arise from spans of finite $G$-sets: Given $G$-sets $A$ and $B$, a \emph{span} from $A$ to $B$ is a diagram of the form  
\[ A \leftarrow C\to B\]
where $C$ is also a finite $G$-set and both maps are equivariant. Note that such a diagram is equivalent to a map $C\to B\times A$ as well. An \emph{isomorphism} of spans is a diagram of the form:
\[
\xymatrix @R=1.25pc{ & C\ar[dl]\ar[dr]\ar[dd]^{\cong}\\
A && B\\
&D\ar[ul]\ar[ur]}
\]
  Spans from $A$ to $B$ have a monoidal product given by disjoint union in the source of the two maps of the span, and the morphisms in $\B_G$ are defined  to be the group completion of the abelian monoid of isomorphism classes of spans. Composition is given by pullback of spans. These morphisms make  $\B_G$ into an additive category. 
\end{defn}

 It is evident from the above definition that $\B_G$ can be thought of as a group-completed quotient of a 2-category in which the one-morphisms are actual spans and the two-morphisms are the isomorphisms of spans. In fact, using the skeletal version of $G$-sets above, it is not difficult to define a $\PC$-category version.

\begin{defn} The $\PC$-category $\BGcat$ has  finite $G$-sets as objects. If $A$ and $B$ are finite $G$-sets, the permutative category of morphisms $\BGcat(A,B)$ is the category of spans of finite $G$-sets from $A$ to $B$  and isomorphisms of spans.  The permutative structure is given by disjoint union. Composition is given by pullback, where we make the following explicit choices.  Let 
\[A\xto{f} B \xleftarrow{f'} C\]
be a diagram of $G$-maps.  If $f$ is the identity, choose the pullback $B\times_B C=C$ and if $f'$ is the identity, choose the pullback $A\times_B B=A$. These specific choices of pullbacks along identity maps make composition strictly unital.  If neither $f$ nor $f'$ is the identity, choose the pullback $A\times_B C$ to be the subset of $A\times C$ picked out by the pullback condition, where the set $A\times C$ is given lexicographical ordering.  Together these choices ensure that composition in $\BGcat$ is strictly associative.
\end{defn}

Thus $\B_G$ is produced from $\BGcat$ by first taking the quotient by the 2-morphisms and then taking the group completion of the resulting abelian monoids of morphisms between objects.

Since $\BGcat$ is a $\PC$-category, Theorem \ref{KthryofPCcat} allows us to produce a spectrally enriched version  $\K_\bullet\BGcat$ by applying $\K$-theory to the morphism permutative categories.
\begin{defn} Let $\BGspec$ denote the spectrally enriched Burnside category $\K_\bullet\BGcat$.  
\end{defn}

This category is the spectrally-enriched category used to model genuine $G$-spectra in the work of Guillou--May.

\begin{thrm}[\cite{GM2011}]\label{mainthrmofGuillouMay}
Let $G$ be a finite group.  The category of genuine $G$-spectra is Quillen equivalent to the category of spectrally-enriched functors $\BGspec^\op\to \Spec$.
\end{thrm}

We can now state Theorem \ref{thrm:inftloopmachine} precisely.
\begin{thrm}\label{precisestatementofmainthrm}
Let $G$ be a finite group.  There is a functor
\[ \eqmach\colon \Fun_{\PC}(\BGcat^\op,\perm)\to \Fun_{\Spec}(\BGspec^\op,\Spec)\]
from the category of $\PC$-functors $\BGcat^\op\to \perm$ and $\PC$-natural transformations to the category of spectral functors $\BGspec^\op\to \Spec$ and spectral natural transformations with the following property. For every $\PC$-functor $X\colon \BGcat^\op\to \perm$ and every finite $G$-set $A$, the spectral functor $\eqmach(X)$ takes $A\in \BGspec$ to the spectrum $\K(X(A))$. 
\end{thrm}
Since $\Fun_{\Spec}(\BGspec^\op,\Spec)$ is Quillen equivalent to the category of genuine $G$-spectra,  Theorem \ref{precisestatementofmainthrm} should be regarded as a functorial construction of $G$-spectra from categorical data.

Our construction provides control over the Mackey functor homotopy groups of an output $G$-spectrum in terms of the categories used to build it.  Recall that $\pi_0\colon \Spec\to \ab$ is a lax monoidal functor and thus induces a change of enrichment functor 
\[ (\pi_0)_\bullet\colon \Spec\textrm{-}Cat \to \ab\textrm{-}Cat \]
which replaces function spectra with their 0th homotopy groups.

 From the proof of the main theorem of Guillou and May, one can readily show the following.
\begin{prop}\label{GM:pi0} Let $Y$ be a spectral functor $\BGspec^\op\to \Spec$, thought of as a genuine $G$-spectrum.  Then the Mackey functor $\underline{\pi}_n(Y)$ is the composite functor
\[\B_G^\op\cong (\pi_0)_\bullet\BGspec^\op\xto{(\pi_0)_\bullet Y} \HoSpec\xto{\pi_n} \ab\]
\end{prop}
Here we are using the facts that applying $(\pi_0)_\bullet$ to $\BGspec$ produces the usual Burnside category and that $(\pi_0)_\bullet\Spec$ is the homotopy category of spectra.

Together with Proposition \ref{pi0works}, this proposition allows us to identify the homotopy Mackey functors of $G$-spectra produced by our construction.
\begin{prop}\label{prop:thehomotopygroupmackeyfunctors} Let $X\col \BGcat^\op\to \perm$ be a $\PC$-functor.  Let $\eqmach X$ be the spectral functor $\BGspec^\op\to \Spec$ produced from $X$ by Theorem \ref{precisestatementofmainthrm}.  Then the Mackey functor homotopy group $\underline{\pi}_n(\eqmach X)$  is given by the composite  
\[\B_G^\op \cong (\pi_0)_\bullet (\K_\bullet(\BGcat^\op))\xto{(\pi_0)_\bullet\K_\bullet X} (\pi_0)_\bullet(\K_\bullet \perm) \xto{(\pi_0)_\bullet\kobj} \HoSpec\xto{\pi_n} \ab\]
In particular, $\underline{\pi}_0(\eqmach X)$ is given by
\begin{equation*}
\underline{\pi}_0(\eqmach X)(G/H)= Gr([obX(G/H)])
\end{equation*}
where for a permutative category $\A$, the notation $Gr([ob\A])$ denotes the group completion of the connected components of objects of $\A$.
\end{prop}

\begin{proof}[Proof of Theorem \ref{precisestatementofmainthrm}]
Let $X\col \BGcat^\op\to \perm$ be a $\PC$-functor.  The spectral functor $\eqmach(X)$ is defined as follows. We apply the functor $\K_\bullet$ of  Theorem \ref{KthryofPCcat}, to produce a spectrally-enriched functor
\[\K_\bullet X\colon \BGspec^\op\to \K_\bullet(\perm).\]
 Composition with the spectrally enriched functor $\kobj\col \K\perm\to \Spec$ then yields a spectrally-enriched functor $\BGspec^\op\to \Spec$.  By Theorem \ref{mainthrmofGuillouMay}, such a functor may be viewed as a genuine $G$-spectrum.  
 
This process is by definition functorial in natural transformations of $\PC$-functors.
\end{proof}

\begin{proof}[Proof of Proposition \ref{prop:thehomotopygroupmackeyfunctors}]
The statement about $\underline{\pi}_n(Y)$ follows directly from applying Proposition \ref{GM:pi0} to the functor $Y= \kobj\circ \K_\bullet X$ constructed in Theorem \ref{precisestatementofmainthrm}. The 0th homotopy groups behave as stated because the composite $\pi_0\circ \K$ is given by group completion.
\end{proof}

\section{Eilenberg--MacLane spectra}\label{EMspec}
Our motivating example is the construction of equivariant Eilenberg--MacLane spectra for Mackey functors.  These can be constructed in a particularly straightforward fashion, as we discuss below.  This construction of Eilenberg--MacLane spectra for Mackey functors  enjoys several nice properties.  Unlike a naive approach via killing homotopy groups, our construction is functorial in natural transformations of Mackey functors.   It also treats the transfer and restriction maps in precisely the same way.  Constructions via an equivariant  Dold--Thom theorem, such as those of dos Santos and Nie \cite{dosSantosNie}, build transfer maps in a more complicated way from those of restriction.  This breaks the self-duality symmetry that Mackey functors naturally have.  Our infinite loop space machine inherently preserves this self-duality.

Recall that a Mackey functor is an additive contravariant functor from the Burnside category $\B_G$ to abelian groups.
 Any such functor in fact determines a $\PC$-functor from $\BGcat^\op\to \perm$, and thus produces a genuine $G$-spectrum via the construction of Section \ref{infiniteloops}.

\begin{thrm}\label{EMspecthrm} The construction of Theorem \ref{precisestatementofmainthrm} produces Eilenberg--MacLane spectra for Mackey functors.
\end{thrm}

There are two steps involved in proving this claim. First, we must show that a Mackey functor produces the correct input for our construction and then we must show that the machine outputs an Eilenberg--MacLane spectrum.

\begin{lemma}\label{Mackeyfunctorlemma}A Mackey functor $M\colon \B_G^\op\to \ab$ determines a $\PC$-functor $M\colon \BGcat^\op\to \perm$.
\end{lemma}
\begin{proof}
Let $M\colon \B_G^\op\to \ab$ be a Mackey functor. As discussed in Section \ref{infiniteloops}, $\B_G$ is the quotient of the 2-category $\BGcat$.  We can regard $\ab$ as a 2-category with only identity 2-cells.    Hence $M$ determines a $2$-functor
\[ \BGcat^\op\fto{\mathrm{quot}} \B_G^\op \fto{M} \ab,\]
which, by abuse of notation,  we also call $M$.

Since $\ab$ is an abelian category, all hom-sets in $\ab$ are actually abelian groups.  We can regard each hom-abelian group as a permutative category whose objects are the elements of the abelian group and whose  monoidal product is given by addition. This makes $\ab$ into a $\PC$-category.  In fact, by regarding the object abelian groups as permutative categories as well, we may regard $\ab$ as a subcategory of the $\PC$-category $\perm$.  Since any Mackey functor is required to be additive, $M$ extends to a $\PC$-functor $M\col \BGcat^\op\to \ab\subset \perm$.  
\end{proof}

\begin{rem}
Because $\ab$ is a $\PC$-category with no non-identity 2-cells, any $\PC$-functor from $\BGcat^\op$ to $\ab$ must factor through the quotient $\B_G^\op$. 
\end{rem}

 Lemma \ref{Mackeyfunctorlemma} puts us in the situation of Section \ref{infiniteloops} and we apply Theorem \ref{precisestatementofmainthrm} to produce a $G$-spectrum $\eqmach M\col \BGspec^\op\to \Spec$.  We now need only check that $\eqmach M$ is an Eilenberg--MacLane spectrum for $M$.

\begin{proof}[Proof of Theorem \ref{EMspecthrm}]
As in the previous lemma, $M$ determines a functor $M\colon \BGcat^\op\to \perm$.   Theorem \ref{precisestatementofmainthrm} produces a genuine $G$-spectrum $\eqmach M\col \BGspec^\op\to \Spec$ from this data.  In order to show that this is indeed an Eilenberg--MacLane spectrum for $M$, we must check that its homotopy groups are correct using Proposition \ref{GM:pi0}.
This proposition asserts that $\underline{\pi}_n(\eqmach M)$ is the Mackey functor
\[ \B_G^\op\cong (\pi_0)_\bullet\K_\bullet \BGcat\xto {(\pi_0)_\bullet\circ\K_\bullet M} (\pi_0)_\bullet\K_\bullet \perm \xto{(\pi_0)_\bullet\kobj} \HoSpec \xto{\pi_n} \ab\]

On objects,  the functor $\eqmach M\col \BGspec^\op\to \Spec$ is easy to understand:  the image of a finite $G$-set $A\in \BGspec$ is the $K$-theory spectrum of the abelian group $M(A)$, which is a nonequivariant Eilenberg--MacLane spectrum $H(M(A))$.  This implies that its homotopy groups are 
\[\underline{\pi}_n(\eqmach M)(A)=\pi_n(H(M(A)))=\begin{cases} M(A) &\textrm{if\ } n=0\\
0 &\textrm{otherwise.} 
\end{cases}
\]
Thus $\underline{\pi}_n(\eqmach M)$ is the zero Mackey functor for $n\neq 0$. 

 We need only show that $\underline{\pi}_0(\eqmach M)$ is the original Mackey functor $M$. We have just identified the group $\underline{\pi}_0(\eqmach M)(A)$ as $M(A)$ and so we need only show that the maps relating these groups are those of the Mackey functor $M$.  Let 
\[f =  B \leftarrow C\to A\]
be a span.  The morphisms in $\B_G(B,A)$ are the group completion of isomorphism classes of such spans, and thus their image under the composite above  is determined by understanding the image of spans.  By the construction of $M\col \BGcat^\op\to \perm$, the span $f$ goes to the strictly unital monoidal functor given on objects by the group homomorphism $M(f)\colon M(A)\to M(B)$; all morphisms in the source and target categories are the identity.  Proposition \ref{pi0worksspectralevel} implies that this functor is mapped under $(\pi_0)_\bullet\kobj$ to the homotopy class of the  induced map of $K$-theory spectra. In this case, the $K$-theory spectra are the Eilenberg--MacLane spectra $H(M(A))$ and $H(M(B))$ and thus the homotopy class of the map $\K M(f)$ is determined by its image on $\pi_0$.  By Corollary \ref{pi0works}, we conclude that $\pi_0[\K M(f)]\col \pi_0H(M(A))\to \pi_0 H( M(B))$ is the desired group homomorphism $M(f)$.
\end{proof}

\section{Suspension spectra of finite $G$-sets}\label{sect:suspspec}

Using the main construction of Section \ref{infiniteloops}, we also produce equivariant suspension spectra for finite $G$-sets, and in particular, the $G$-equivariant sphere spectrum.  These suspension spectra are the results of applying our machine to the representable functors $\BGcat^\op\to \perm$ given by the finite $G$-sets $X\in \BGcat$.  In their paper \cite{GM2011}, Guillou and May identify suspension spectra of finite $G$-sets as representable functors $\BGspec^\op\to \Spec$, and so our construction of these suspension spectra is not a surprising result.  Rather, it may be thought of as a sort of consistency check which demonstrates that our construction agrees with pre-existing understandings of equivariant spectra in terms of fixed points.

\begin{defn}\label{defnofsuspfunct} Let $S_X\colon \BGcat^\op \to \perm$ be the representable functor given by the finite $G$-set $X\in \BGcat$.  That is, for any $A\in \BGcat$,  
\[S_X(A)=\BGcat^\op(X,A)=\BGcat(A, X)\]
which is a permutative category since $\BGcat$ is a $\PC$-category. On morphisms, this functor is given by the unique strictly unital symmetric monoidal functor 
 \[ S_X \colon \BGcat^\op(A,B) \to \perm (\BGcat^\op(X,A),\BGcat^\op(X,B))\]
which  Lemma \ref{lemma:closed} produces from the bilinear composition map
\[ \comp\colon (\BGcat^\op(A,B),\BGcat^\op(X,A))\to \BGcat^\op(X,B).\]
\end{defn}

\begin{prop}\label{suspenfunctorwelldefined} The functor $S_X\colon \BGcat^\op \to \perm$ is a well-defined $\PC$-functor.
\end{prop}
\begin{proof} This fact follows from the universal property of evaluation given in Lemma \ref{lemma:closed} via a direct analogue of  the standard proof of the existence of representable functors for categories enriched in a monoidal category $\V$.  Alternately, the authors have checked that the diagrams of Definition \ref{defn:enrfunctor} commute using the explicit definition of composition and units in the category $\BGcat$.  Either way, the key to the proof is that composition in $\BGcat$ is associative.
\end{proof}

\begin{rem}
Proposition \ref{suspenfunctorwelldefined} is the multicategorical version of the standard fact that representable functors exist for categories enriched in closed symmetric monoidal categories. As mentioned in Remark \ref{rem:otherwaystoseeclosedmonoidalstuff}, this type of closure property can be viewed as arising from the embedding of $\permmc$ into the closed symmetric monoidal category of based multicategories of \cite{EM2009}.
\end{rem}

We can apply our construction to the functor $S_X(-)$. 
\begin{thrm}\label{identifysuspspec} 
Suppose $X$ is a finite $G$-set. Let $\eqmach S_X\colon \BGspec^\op\to \Spec$ be the functor $\kobj\circ \K_\bullet S_X$ produced from $S_X$ by Theorem \ref{precisestatementofmainthrm}.  Then $\eqmach S_X$ represents the suspension spectrum $\Sigma^\infty_G(X_+)$. 
\end{thrm}

As mentioned at the beginning of this section, we prove this theorem by comparing $\eqmach S_X$ to the suspension spectrum $\Sigma^\infty_G(X_+)$ as characterized by Guillou and May.

\begin{prop}[{\cite[\S 2.5]{GM2011}}]\label{GMsuspensionspectra}
The suspension spectrum $\Sigma^\infty_G(X_+)$ corresponds to the presheaf of spectra
\[\BGspec^\op\to \Spec\]
represented by $X\in \BGspec$.
\end{prop}

\begin{proof}[Proof of Theorem \ref{identifysuspspec}]
By Prop \ref{GMsuspensionspectra}, we need only compare our spectral functor $\eqmach S_X$ to the spectral functor $\BGspec^\op(X,-)$  represented by $X$. By definition these functors are the same on objects: they both take an object $A\in \BGspec^\op$ to the spectrum $\K\BGcat^\op(X,A)=\K\BGcat(A,X)$.  We must show that the maps
\begin{equation}\label{suspensionspecmorphismfunct}\begin{split}\eqmach S_X\colon &\BGspec^\op(A,B)\to F(\BGspec^\op(X,A),\BGspec^\op(X,B))\\
\BGspec^\op(X,-)\colon &\BGspec^\op(A,B) \to F(\BGspec^\op(X,A),\BGspec^\op(X,B))
\end{split}
\end{equation}
which these functors induce on morphisms are also the same.

By definition, the map $\eqmach S_X$ is the composite
\[\K\BGcat^\op(A,B)\xto{\K_\bullet S_X} \K\perm(\BGcat^\op(X,A),\BGcat^\op(X,B))\xto{\kobj} F(\BGspec^\op(X,A),\BGspec^\op(X,B))\]
  Recall that $\kobj$ is defined in terms of its adjoint.  The map 
\[\BGspec^\op(X,-)\col \BGspec^\op(A,B)\to F(\BGspec^\op(X,A),\BGspec^\op(X,B))\]
 is defined to be the adjoint of 
\[\BGspec^\op(A,B)\sma\BGspec^\op(X,A) \xto{\comp} \BGspec^\op(X,B)\]
 Thus the maps on morphisms of (\ref{suspensionspecmorphismfunct}) agree if  the diagram
\begin{equation*}\label{boilsdowntothisdiag}
\xymatrix@C=9ex{
\K\BGcat^\op(A,B)\sma \K\BGcat^\op(X,A)\ar[dr]_{\comp}\ar[r]^-{\K_\bullet S_X\sma \id} &\K\perm(\BGcat^\op(X,A),\BGcat^\op(X,B))\sma\K\BGcat^\op(X,A) \ar[d]^{\K(ev)} \\
& \K\BGcat^\op(X,B)
}
\end{equation*}
commutes.  Here we have replaced $\BGspec^\op(-,-)$  by  $\K\BGcat^\op(-,-)$  since  the latter spectrum is the definition of the former.  The arrow labeled $\comp$ is composition in the spectrally enriched category $\BGspec^\op$.  Since $\BGspec=\K_\bullet\BGcat$, this arrow is given by applying the multifunctor $\K$ to the composition morphism in the $\PC$-category $\BGcat^\op$.  Hence, the previous diagram is the image under $\K$ of the diagram of bilinear functors of permutative categories
\[
\xymatrix@C=4pc{
\left(\BGcat^\op(A,B), \BGcat^\op(X,A) \right) \ar[r]^-{(S_X,\id)}\ar[dr]_{\comp} & \left( \perm(\BGcat^\op(X,A), \BGcat^\op(X,B)), \BGcat^\op(X,A)\right)\ar[d]^{ev} \\
&\BGcat^\op(X,B).
}
\]
This diagram commutes by definition: it is an instance of the universal property of evaluation of Lemma  \ref{lemma:closed} which is used to define $S_X$.

Therefore, the functor $S_X\col \BGcat^\op\to \perm$ produces the suspension spectra $\Sigma^\infty_G(X_+)$ under application of the construction of Theorem \ref{precisestatementofmainthrm}.
\end{proof}

\bibliographystyle{alpha}
\bibliography{references}

\end{document}